\documentclass[11pt,a4paper]{amsart}

\usepackage{amssymb,amsmath,amsthm}
\usepackage{hyperref}
\usepackage{graphicx}
\usepackage{xcolor}
\usepackage{graphicx}
\usepackage{url}
\usepackage{enumerate}

\newcommand{\vx}{\mathbf{x}}
\newcommand{\ve}{\mathbf{e}}
\newcommand{\vy}{\mathbf{y}}
\newcommand{\va}{\mathbf{a}}

\newcommand{\vv}{\mathbf{v}}
\newcommand{\vw}{\mathbf{w}}
\newcommand{\vu}{\mathbf{u}}
\newcommand{\vz}{\mathbf{z}}
\newcommand{\vg}{\mathbf{g}}
\newcommand{\vh}{\mathbf{h}}
\newcommand{\vnull}{\mathbf{0}}
\newcommand{\valpha}{\boldsymbol\alpha}
\newcommand{\veps}{\boldsymbol\varepsilon}

\definecolor{author}{rgb}{0.5,0.5,0.0}
\definecolor{def}{rgb}{0.0,0.5,0.5}
\definecolor{high}{rgb}{0.1,0.2,0.5}
\definecolor{darkgreen}{rgb}{0.0,0.4,0.0}
\definecolor{darkred}{rgb}{0.4,0.0,0.0}

\newcommand{\diag}{\mathrm{diag}}
\newcommand{\Kn}{\mathcal{K}^n}
\newcommand{\Kon}{\mathcal{K}_o^n}
\newcommand{\Kgn}{\mathcal{K}_c^n}

\DeclareMathOperator{\vol}{vol}

\DeclareMathOperator{\inte}{int}
\DeclareMathOperator{\conv}{conv}

\DeclareMathOperator{\affo}{aff_o}
\DeclareMathOperator{\lin}{lin}
\newcommand{\R}{\mathbb{R}}
\newcommand{\Z}{\mathbb{Z}}
\newcommand{\N}{\mathbb{N}}
\newcommand{\Q}{\mathbb{Q}}
\newcommand{\C}{\mathbb{C}}

\newcommand{\bone}{\mathbf{1}}

\DeclareMathOperator{\sur}{F}

\newcommand{\ip}[2]{\left\langle #1,#2\right\rangle}

\newcommand{\dlat}{\mathrm{d}}
\newcommand{\e}{\varepsilon}
\DeclareMathOperator{\lE}{g}
\newcommand{\D}{\mathcal{D}}
\DeclareMathOperator{\cen}{cen}

\DeclareMathOperator{\GL}{GL}

\DeclareMathOperator{\Dee}{D}
\DeclareMathOperator{\Dmin}{D}

\newtheorem{theorem}{Theorem}[section]
\newtheorem{corollary}[theorem]{Corollary}
\newtheorem{lemma}[theorem]{Lemma}
\newtheorem{example}[theorem]{Example}
\newtheorem{remark}[theorem]{Remark}

\newtheorem{conjecture}{Conjecture}[section]
\newtheorem{proposition}[theorem]{Proposition}

\newtheorem{fact}[theorem]{Fact}

\numberwithin{equation}{section}

\begin{document}

\title[On extensions of Minkowski's theorem on successive minima]{On extensions of \\Minkowski's theorem on successive minima}

\author{Martin Henk}
\address{Fakult\"at f\"ur Mathematik, Otto-von-Guericke
Universit\"at Mag\-deburg, Universit\"atsplatz 2, D-39106 Magdeburg,
Germany} \email{martin.henk@ovgu.de}

\author{Matthias Henze}

\address{Institut f\"ur Informatik, Freie Universit\"at Berlin, Takustrasse 9, 14195 Berlin, Germany}
\email{matthias.henze@fu-berlin.de}

\author{Mar\'\i a A. Hern\'andez Cifre}
\address{Departamento de Matem\'aticas, Universidad de Murcia, Campus de
Espinar\-do, 30100-Murcia, Spain} \email{mhcifre@um.es}

\keywords{Minkowski's 2nd Theorem, successive minima, centroid, volume, surface area, lattice
surface area, convex body, rational polytope} \subjclass[2010]{52A40, 52A20,
52B11, 52C07}

\thanks{First and third authors were partially supported by MINECO-FEDER project MTM2012-34037.
The second author was partially supported by the DFG projects He 2272/4-1 and RO 2338/5-1.}

\begin{abstract}
Minkowski's 2nd theorem in the Geometry of Numbers provides optimal upper
and lower bounds for the volume of a $o$-symmetric convex body in terms of
its successive minima. In this paper we study extensions of this theorem
from two different points of view: either relaxing the symmetry condition,
assuming that the centroid of the set lies at the origin, or replacing the
volume functional by the surface area.
\end{abstract}

\maketitle

\section{Introduction}

Let $\Kn$ be the set of all convex bodies, i.e., compact convex sets, in
the $n$-dimensional Euclidean space $\R^n$ with non-empty interior. Let
$\ip{\,\cdot}{\cdot}$ and $\|\cdot\|$  be the standard inner product and
the Euclidean norm in $\R^n$, respectively. We denote by $\Kon\subset \Kn$
the set of all $o$-symmetric bodies, i.e., those $K\in\Kn$ satisfying
$K=-K$, and let $\Kgn\subset \Kn$ be the set of all convex bodies with
centroid at the origin, i.e.,
\begin{equation*}
 \cen(K)=\frac{1}{\vol(K)}\int_K\vx\,{\rm d}^n\vx =\vnull.
\end{equation*}
Here, ${\rm d}^n\vx$ means integration with respect to  the  $n$-dimensional
Lebesgue-measure and $\vol(K)=\int_K {\rm d}^n\vx$ is the volume of $K$.
The surface area of $K\in\Kn$ is denoted  $\sur(K)$, and for general
information on the theory of convex bodies we refer to
\cite{Gr,Schneider:2014td}.

We denote by $\Z^n$ the integer lattice, i.e., the lattice of all points
with integral coordinates in $\R^n$. Then any lattice $\Lambda\subset
\R^n$ of rank $n$ can be obtained as $\Lambda=B\Z^n$ with $B\in\GL(n,\R)$,
and the determinant of the lattice is defined as $\det\Lambda=|\det B|$.
As general references for lattices we refer to \cite{Gr,GL}.

For $K\in \Kon\cup \Kgn$  and a lattice $\Lambda$ of rank $n$, let
\begin{equation*}
 \lambda_i(K,\Lambda)=\min\bigl\{\lambda>0 : \dim(\lambda K\cap \Lambda)\geq
 i\bigr\}
\end{equation*}
be the $i$-th successive minimum of $K$ with respect to $\Lambda$, $1\leq
i\leq n$. Minkowski's 2nd theorem on successive minima
\cite{minkowski1896geometrie} (cf.~\cite{Gr}) states that for $K\in\Kon$,
\begin{equation}
\frac{1}{n!}\prod_{i=1}^n \frac{2}{\lambda_i(K,\Lambda)} \leq
\frac{\vol(K)}{\det\Lambda} \leq \prod_{i=1}^n
\frac{2}{\lambda_i(K,\Lambda)}.
\label{eq:minkowski}
\end{equation}
Both bounds are best possible; for instance, for $\Lambda=\Z^n$, the upper
bound is attained for the cube $C_n=[-1,1]^n$ and the lower bound for its
polar body, the cross-polytope $C_n^\star=\conv\{\pm\ve_i : 1\leq i\leq
n\}$, where $\ve_i$ denotes the $i$-th canonical unit vector, and $\conv
S$ is the convex hull of a set~$S$.

Other special convex bodies that will appear throughout the paper are the
standard simplex $S_n=\conv\{\vnull,\ve_1,\ldots,\ve_n\}$ and its
homothetic copy $T_n=-\bone+(n+1)S_n$, where $\bone=(1,\dots,1)^\intercal$
is the all-one-vector.

It is well known that via the difference body $\D K=K-K\in\Kon$,
Minkowski's results \eqref{eq:minkowski} can be generalized  to arbitrary
bodies (see, e.g., \cite[p.~59]{GL}):
\begin{equation}
\frac{1}{n!}\prod_{i=1}^n \frac{1}{\lambda_i(\D K,\Lambda)} \leq
\frac{\vol(K)}{\det\Lambda} \leq \prod_{i=1}^n
\frac{1}{\lambda_i(\D K,\Lambda)}.
\label{eq:minkowski_gen}
\end{equation}
The upper bound is a combination of the upper bound in
\eqref{eq:minkowski} and the Brunn-Minkowski inequality (see e.g.
\cite[Thm.~8.1]{Gr}). The lower bound stems from  the following well-known
fact (see~\cite[Thm.~2]{betkehenk1993approx} or~\cite{groemer1966zus}).
\begin{remark}
Let $\vv_1,\vw_1,\ldots,\vv_n,\vw_n\in K$. Then, the volume of $K$ is at
least the volume of the $o$-symmetric cross-polytope $\conv\bigl\{\pm
(1/2) (\vv_i-\vw_i) : 1\leq i\leq n\bigr\}$. \label{obs:cross}
\end{remark}

In particular,  both bounds in \eqref{eq:minkowski_gen} can only be
realized for $K\in \Kon$, and so they do not provide much more information
than Minkowski's original result for $o$-symmetric convex bodies.
Therefore, we are interested in a variant of~\eqref{eq:minkowski} that
does not rely on the symmetrization~$\D K$. As a first result we obtain the
following lower bound whose proof is given in Section~\ref{s:non-sym}.

\begin{theorem}\label{thm:vol_main}
Let $K\in\Kgn$ and let $\Lambda$ be a lattice of rank $n$. Then
\begin{equation*}
   \frac{n+1}{n!}\prod_{i=1}^n\frac{1}{\lambda_i(K,\Lambda)} \leq \frac{\vol(K)}{\det\Lambda}.
\end{equation*}
Equality holds if and only if there are positive numbers
$\mu_1,\ldots,\mu_n>0$ such that
$K=\conv\bigl\{\mu_1\vz_1,\ldots,\mu_n\vz_n,-\left(\mu_1\vz_1+\dots+\mu_n\vz_n\right)\bigr\}$,
where $\{\vz_1,\ldots,\vz_n\}$ is a basis of~$\Lambda$.
\end{theorem}

A corresponding upper bound on the volume of $K\in\Kgn$ immediately
relates to a longstanding conjecture of Ehrhart~\cite{Ehr2} on the maximal
volume of a convex body $K\in\Kgn$ whose interior is free from non-zero
lattice points. In this context, the best-known bound is based on a result
of Milman and Pajor~\cite{Milman:2000th} showing that $\vol(K)\leq
2^n\vol\bigl(K\cap(-K)\bigr)$ for $K\in\Kgn$. Hence with
\eqref{eq:minkowski} applied to $K\cap (-K) \subseteq K$, we find
\begin{equation}
\frac{\vol(K)}{\det\Lambda} \leq 4^n\prod_{i=1}^n
\frac{1}{\lambda_i(K,\Lambda)},
\label{eq:milmanpajorehrhart}
\end{equation}
and in view of Ehrhart's conjecture (see Conjecture \ref{conj:ehrhart})
the optimal factor is conjectured to be $(n+1)^n/n!$ instead of $4^n$. In
Propositions~\ref{prop:ehr_dim2} and~\ref{prop:ehr_simpl}, we verify this
conjecture for the special cases $n=2$ and simplices of arbitrary
dimension, respectively.

Another direction of extending Minkowski's 2nd theorem is to replace the
volume functional by other functionals, for instance, the lattice point
enumerator (see, e.g., \cite{Betke:1993jn, Malikiosis:2010tb,
Malikiosis:2010vi}) or the intrinsic volumes (see, e.g.,
\cite{henk1990inequ, Schnell1993}). Here we are interested in inequalities
analogous to \eqref{eq:minkowski} for the surface area. In
\cite{henk1990inequ} it was shown $\sur(K)/\vol(K)> \lambda_n(K,\Z^n)$ for
$K\in\Kon$, and with the lower bound in \eqref{eq:minkowski} we get
\begin{equation*}
   \sur(K)> \frac{2^{n}}{n!} \prod_{i=1}^{n-1}\frac{1}{\lambda_i(K,\Z^n)}.
\end{equation*}
In order to present our improvement on this bound,
we need the notation of the elementary symmetric functions
\begin{equation*}
 \sigma_k(\rho_1,\ldots,\rho_n)=\sum_{\substack{J\subseteq\{1,\ldots,n\}\\\#J=k}}\,\prod_{i\in J}\rho_i,
\end{equation*}
for $k\in\{1,\ldots,n\}$, and real numbers $\rho_1,\ldots,\rho_n$.

\begin{theorem}\label{thm:surface}
Let $K\in\Kon$ and let $\lambda_i=\lambda_i(K,\Z^n)$, $1\leq i\leq n$.
Then
\begin{equation*}
   \sur(K)\geq \frac{2^{n}}{(n-1)!}
   \sqrt{\sigma_{n-1}\left(\lambda_1^{-2},\dots,\lambda_n^{-2}\right)},
\end{equation*}
and equality holds if and only if
$K=\diag(\lambda_1^{-1},\dots,\lambda_n^{-1})C_n^\star$,
where $\diag(\cdot)$ denotes the diagonal matrix.
\end{theorem}

The proof of this result is given  in Section~\ref{s:F(K)}.
Generalizations to arbitrary lattices are not so straightforward as those
for the volume functional because the surface area is not
$\mathrm{SL}(n,\R)$-invariant. Still we obtain meaningful results in the
general situation that are presented in Theorem~\ref{cor:surface_arblat}.
We also note  that the above inequality has the same structure as the one
in~\cite[Thm.~1.3]{henkcifre2008intrinsic}, where the surface area is
related to the successive inner radii of a convex body.

In general, we cannot expect to find upper bounds on $\sur(K)$, or on the
quotient $\sur(K)/\vol(K)$,  in terms of $\lambda_i(K,\Z^n)^{-1}$ as
Example \ref{exam:no_ub_sur} shows.

Hence, in order to obtain upper bounds, the convex bodies need to have
more  lattice structure,  and this leads to the class of  rational
polytopes. Here,~$P$ is called a rational polytope if all its vertices lie
in $\Q^n$.  For basic facts and notions about polytopes we refer to
\cite{Ziegler:1995gl}. Given a rational polytope $P\in\Kn$ and a facet
$F$, we denote by $\affo(F)$ the $(n-1)$-dimensional linear subspace
parallel to the affine hull of $F$. We observe, by the rationality of~$P$,
that the intersection $\affo(F)\cap\Z^n$ is an $(n-1)$-dimensional
lattice. With this notation, the lattice surface area can be described as
\[
\lE_{n-1}(P)=\frac12\sum_{F\textrm{ facet of
}P}\frac{\vol_{n-1}(F)}{\det\bigl(\affo(F)\cap\Z^n\bigr)},
\]
where  $\vol_{n-1}(\cdot)$ is the $(n\!-\!1)$-dimensional volume in
$\R^{n-1}$.
The notation $\lE_{n-1}(P)$ is taken from Ehrhart theory,
where the lattice surface area of a lattice
polytope $P$, i.e., all vertices lie in~$\Z^n$, appears as the coefficient
of order $n-1$ in its Ehrhart polynomial
\begin{equation}
\#(kP\cap\Z^n)=\sum_{i=0}^n\lE_i(P)k^i,\quad k\in\N.
\label{eq:ehrhartpoly}
\end{equation}
For details and more information on Ehrhart theory, we refer
to~\cite{beckrobins2007computing}.

For $o$-symmetric lattice polytopes, and actually for $o$-symmetric rational
polytopes, it was shown  \cite[Thm.~1.2]{henkschuerwills2005ehrhart}
that
\begin{align}
\frac{\lE_{n-1}(P)}{\vol(P)}&\leq\frac12\sum_{i=1}^n\lambda_i(P,\Z^n).\label{thm:HSW_cs}
\end{align}
Equality holds, for example, for $C_n$ and $C_n^\star$.
Here we extend and complement  this result by providing  bounds
for  all rational polytopes as well as  for rational  polytopes with
centroid at the origin.

\begin{theorem} \ \label{thm:HSW_nonsymmetric}
\begin{enumerate}[i)]
 \item Let $P\in\Kn$ be a rational polytope. Then
\[
\frac{\lE_{n-1}(P)}{\vol(P)}\leq\frac{n+1}{2}\sum_{i=1}^n\lambda_i(\D
P,\Z^n),
\]
and the standard simplex $S_n$ shows that the inequality is best possible.
 \item Let $P\in\Kgn$ be a rational polytope and let $n\geq2$. Then
\[
\frac{\lE_{n-1}(P)}{\vol(P)}<\frac{n}{2}\sum_{i=1}^n\lambda_i(P,\Z^n).
\]
\end{enumerate}
\end{theorem}

In Section~\ref{s:latticeF(K)}, we discuss the proofs of these results, and
moreover we show that the factor $n/2$ in the second inequality is almost
tight. Now combining  the above bounds on
$\lE_{n-1}(P)/\vol(P)$ with the upper bounds in \eqref{eq:minkowski},
\eqref{eq:minkowski_gen}, or \eqref{eq:milmanpajorehrhart}, we
immediately get
\begin{corollary}
Let $P\in\Kn$ be a rational polytope.
\begin{enumerate}[i)]
\item Then
\begin{equation*}
\lE_{n-1}(P)\leq\frac{n+1}{2}\,\sigma_{n-1}\left(\frac{1}{\lambda_1(\D
    P,\Z^n)},\ldots,\frac{1}{\lambda_n(\D P,\Z^n)}\right).
\end{equation*}
\item If $P\in\Kgn$, then
\begin{equation*}
\lE_{n-1}(P)<4^n\frac{n}{2}\,\sigma_{n-1}\left(\frac{1}{\lambda_1(P,\Z^n)},\ldots,\frac{1}{\lambda_n(P,\Z^n)}\right).
\end{equation*}
\item If $P\in\Kon$, then
\begin{equation*}
\lE_{n-1}(P)\leq 2^{n-1}\,\sigma_{n-1}\left(\frac{1}{\lambda_1(P,\Z^n)},\ldots,\frac{1}{\lambda_n(P,\Z^n)}\right).
\end{equation*}
\end{enumerate}
\end{corollary}
However, only the last inequality is best possible, which has been pointed
out before in~\cite{henkschuerwills2005ehrhart}. Further immediate
consequences of Theorem \ref{thm:HSW_nonsymmetric} are  relations between
the roots of the Ehrhart polynomial --  when we regard the right hand side
of \eqref{eq:ehrhartpoly} as a formal polynomial in a complex variable --
and the successive minima (cf.~Corollary~\ref{cor:roots_appl}). Those kind
of relations were  the main motivation for \eqref{thm:HSW_cs} in
\cite{henkschuerwills2005ehrhart}.

Finally, we remark that in contrast to the surface area, we now cannot expect
lower bounds on $\lE_{n-1}(P)$ in terms of the successive minima as
shown in Example~\ref{exam:no_lb_lsur}.

\section{Volume bounds for $K\in\Kgn$}\label{s:non-sym}

In this section, we discuss a variant of \eqref{eq:minkowski} for the
class of convex bodies having their centroid at the origin, i.e., for $K\in\Kgn$.
A basic and beautiful result in this context is
Gr\"unbaum's halfspace theorem. For a hyperplane~$H$, we denote by $H^+$
and $H^-$ the two associated halfspaces.

\begin{theorem}[Gr\"unbaum,~\cite{gruenbaum1960parts}]\label{thm:gruenbaum}
Let $K\in \Kn$ and let $H^+$ be a halfspace containing the centroid of $K$. Then
\begin{equation*}
 \vol(K\cap H^+)\geq \left(\frac{n}{n+1}\right)^n\vol(K).
\end{equation*}
\end{theorem}

Our first aim is to prove Theorem~\ref{thm:vol_main}, which is
an immediate consequence of the following lemma.
\begin{lemma}
Let $K\in\Kgn$ and let $\vu_1,\dots,\vu_n\in K$ be linearly independent.
Then
\begin{equation*}
  \vol(K)\geq \frac{n+1}{n!}\, \bigl|\det(\vu_1,\dots,\vu_n)\bigr|.
\end{equation*}
Equality holds if and only if $K=\conv\bigl\{\vu_1,\dots,\vu_n,
-(\vu_1+\cdots +\vu_n)\bigr\}$. \label{lem:main}
\end{lemma}

\begin{proof}
Via a suitable linear transformation, we may assume
that all the vectors $\vu_i$ have first
coordinate equal to~$-1$ and that $\det(\vu_1,\dots,\vu_n)=1$. For $t\in\R$, let
$H_t=\bigl\{\vx\in\R^n:\ip{\ve_1}{\vx}=t\bigr\}$ be the family of
hyperplanes orthogonal to the first unit vector $\ve_1$.

As in Gr\"unbaum's proof of Theorem~\ref{thm:gruenbaum}, we first apply
Schwarz-sym\-metrization to $K$ with respect to  $\lin\{\ve_1\}$, the
linear hull of $\ve_1$ (see e.g. \cite[Sect.~9.3]{Gr}). Denoting by $B_n$
the $n$-dimensional unit ball, this means that for every $t\in\R$, we
replace $K\cap H_t$ by the ($n-1$)-ball $t\ve_1+r(t)(B_n\cap H_0)$ with
center $t\ve_1$ and having the same volume as $K\cap H_t$, i.e.,
\[
r(t)=\left(\frac{\vol_{n-1}(K\cap
H_t)}{\vol_{n-1}(B_{n-1})}\right)^{1/(n-1)}.
\]
The so created convex body $L$, say, is symmetric with respect to
$\lin\{\ve_1\}$, and we also have $\vol(L)=\vol(K)$ and
$\cen(L)=\vnull$. With  $F=L\cap H_{-1}$, we get by the choice of the
vectors $\vu_i$
\begin{equation}
  \vol\bigl(\conv\{F,\vnull\}\bigr)\geq\vol\bigl(\conv\{\vnull,\vu_1,\dots,\vu_n\}\bigr)=\frac{1}{n!}.
\label{eq:elementar}
\end{equation}
Now let $\widehat{L}=K\cap\bigl\{\vx\in\R^n:\ip{\ve_1}{\vx}\geq -1\bigr\}$
and let $\beta>0$ be such that the pyramid $P=\conv\{F,\beta\,\ve_1\}$ has
the same volume as $\widehat{L}$. Since $\widehat{L}$ and $P$ are
symmetric with respect to $\ve_1$ their centroids are on the line
$\lin\{\ve_1\}$ and so we may write  $\cen(P)=\gamma_P\,\ve_1$,
$\cen(\widehat{L})=\gamma_{\widehat{L}}\,\ve_1$ for suitable numbers
$\gamma_P, \gamma_{\widehat{L}}$ with $\gamma_{\widehat{L}}\geq 0$.  For
pyramids we have  $\vol(P) =
  (n+1)\vol\bigl(\conv\{F,\gamma_P\,\ve_1\}\bigr)$ (see
e.g.~\cite[Sect.~34]{bonnesenfenchel1987theory}) and so in view of
\eqref{eq:elementar}
\begin{equation}
\begin{split}
  \vol(K) = &\vol(L)\geq\vol(\widehat{L})=\vol(P)\\ = &
  (n+1)\vol\bigl(\conv\{F,\gamma_P\,\ve_1\}\bigr)\geq \frac{n+1}{n!}(1+\gamma_P).
\end{split}
\label{eq:pyramid}
\end{equation}
It remains to show $\gamma_P\geq 0$. Actually we will
show~$\gamma_P\geq\gamma_{\widehat{L}}$, which seems to be quite evident.
But since we also want to discuss the equality case, we present a proof.

We may assume $\widehat{L}\neq P$. For $t\in\R$, let $l(t)$ be the radius
of the $(n-1)$-ball $P\cap H_t$, i.e.,
\[
l(t)=\left(\frac{\vol_{n-1}(P\cap
H_t)}{\vol_{n-1}(B_{n-1})}\right)^{1/(n-1)}.
\]
Then $l(t)\ne 0$ if and only if $t\in[-1,\beta)$ and $l(t)$ is an affine
function.

Since $r(-1)=l(-1)$ and $r(t)$ is  concave, $\widehat{L}\ne P$, and
$\vol(P)=\vol(\widehat{L})$, there exists a unique $\alpha\in(-1,\beta)$
with
\begin{equation*}
r(t)>l(t)\text{ for } t\in(-1,\alpha)\quad\text{ and }\quad
l(t)>r(t)\text{ for } t\in(\alpha,\beta).
\end{equation*}
Hence we know
\begin{equation*}
\ip{\ve_1}{\vx}\leq\alpha\text{ for }\vx\in \widehat{L}\setminus
P\quad\text{ and }\quad \ip{\ve_1}{\vx}\geq\alpha\text{ for }\vx\in
P\setminus \widehat{L}.
\end{equation*}
Finally, since $\vol(P)=\vol(\widehat{L})$ it holds
$\vol(\widehat{L}\setminus P)=\vol(P\setminus \widehat{L})$ and so we get
\begin{equation*}
 \begin{split}
\gamma_P &=\int_P\ip{\ve_1}{\vx}\dlat^n\vx
    =\int_{P\setminus \widehat{L}}\ip{\ve_1}{\vx}\dlat^n\vx+\int_{P\cap \widehat{L}}\ip{\ve_1}{\vx}\dlat^n\vx\\
 & >\alpha\vol(P\setminus \widehat{L})+\int_{P\cap \widehat{L}}\ip{\ve_1}{\vx}\dlat^n\vx
    =\alpha\vol(\widehat{L}\setminus P)+\int_{P\cap \widehat{L}}\ip{\ve_1}{\vx}\dlat^n\vx\\
 & >\int_{\widehat{L}\setminus P}\ip{\ve_1}{\vx}\dlat^n\vx
    +\int_{P\cap\widehat{L}}\ip{\ve_1}{\vx}\dlat^n\vx =\int_{\widehat{L}}\ip{\ve_1}{\vx}\dlat^n\vx=\gamma_{\widehat{L}}.
 \end{split}
\end{equation*}
Hence $\gamma_P>\gamma_{\widehat{L}}>0$ as desired, since we have assumed
$\widehat{L}\ne P$.

If we have equality, then \eqref{eq:pyramid} gives $L=\widehat{L}$ and
$\gamma_P=0$, and in view of the above argumentation we must also have
$L=P$. Since we also must have equality in \eqref{eq:elementar}, we
conclude $K\cap H_{-1}=\conv\{\vu_1,\ldots,\vu_n\}$. Let $\vu\in K$ be the
point whose image under the Schwarz-symmetrization is the apex
$\beta\,\ve_1$ of the pyramid. Since $L=P$, we have
$K=\conv\{\vu_1,\ldots,\vu_n,\vu\}$. Finally, since for a simplex the
centroid coincides with the arithmetic mean of its vertices, we get
$\vu=-(\vu_1+\vu_2+\dots+\vu_n)$.
\end{proof}

The proof of Theorem~\ref{thm:vol_main} is now an immediate consequence of
the Lemma above.

\begin{proof}[Proof of Theorem~\ref{thm:vol_main}]
We write $\lambda_i=\lambda_i(K,\Lambda)$ and let $\vz_1,\dots,\vz_n\in\Lambda$ be
linearly independent lattice points such that $\vz_i/\lambda_i\in K$,
$1\leq i\leq n$. Lemma~\ref{lem:main} applied to these vectors gives
\begin{equation*}
  \vol(K)\geq\bigl|\det(\vz_1,\dots,\vz_n)\bigr|\,\frac{n+1}{n!}\prod_{i=1}^n\frac{1}{\lambda_i}
 \geq \det\Lambda\frac{n+1}{n!}\prod_{i=1}^n\frac{1}{\lambda_i},
\end{equation*}
and equality holds if and only if
$\bigl|\det(\vz_1,\dots,\vz_n)\bigr|=\det\Lambda$, i.e.,
$\{\vz_1,\ldots,\vz_n\}$ is a basis of $\Lambda$, and
\[
K=\conv\left\{\frac{1}{\lambda_1}\vz_1,\dots,\frac{1}{\lambda_n}\vz_n,
-\Bigl(\frac{1}{\lambda_1}\vz_1+\dots+\frac{1}{\lambda_n}\vz_n\Bigr)\right\}.
\]
In order to further discuss the equality case, it is no restriction to
assume that $\Lambda=\Z^n$ and $\vz_i=\ve_i$, $1\leq i\leq n$. We write
$\inte M$ to denote the interior of a set $M$. Let
$K=\conv\bigl\{\mu_1\ve_1,\ldots,\mu_n\ve_n,-(\mu_1\ve_1+\dots+\mu_n\ve_n)\bigr\}$
for real numbers $\mu_1,\ldots,\mu_n>0$. Assuming that
$\mu_1\geq\mu_2\geq\dots\geq\mu_n$, we see that
$\inte\bigl((1/\mu_i)K\bigr)\cap\Z^n\subset\lin\{\ve_1,\ldots,\ve_{i-1}\}$
and $\ve_i\in(1/\mu_i)K$, for $1\leq i\leq n$. It means that
$\lambda_i(K,\Z^n)=1/\mu_i$ for $1\leq i\leq n$, and thus $K$ attains
equality.
\end{proof}

As mentioned in the introduction, the question about an upper bound as in
\eqref{eq:minkowski} for $K\in\Kgn$ is strongly related to Ehrhart's
conjecture (see~\cite{Ehr2}, and also \cite{BermanBerndtsson, NillPaffenholz}).

\begin{conjecture}[Ehrhart, \cite{Ehr2}]\label{conj:ehrhart}
Let $K\in\Kgn$ with $\inte K\cap\Z^n=\{\vnull\}$. Then
\begin{equation*}
\vol(K)\leq \frac{(n+1)^n}{n!},
\end{equation*}
and equality holds if and only if $K$ is -- up to unimodular
transformations -- the simplex $T_n$.
\end{conjecture}

Ehrhart \cite{Ehr, Ehr3} proved his conjecture, among others, for two
dimensional convex bodies. Here we follow his approaches in order to
extend his results to successive minima inequalities.

\begin{proposition}\label{prop:ehr_dim2}
Let $K\in\mathcal{K}_c^2$ and let $\Lambda$ be a lattice of rank $2$. Then
\begin{equation*}
 \frac{\vol(K)}{\det\Lambda}\leq \frac{9}{2}\,\frac{1}{\lambda_1(K,\Lambda)}\frac{1}{\lambda_2(K,\Lambda)},
\end{equation*}
and for $\Lambda=\Z^2$, equality holds for the triangle $T_2$.
\end{proposition}

\begin{proof}
As always when dealing with the volume, we may assume $\Lambda=\Z^2$,  and
for short we write $\lambda_i=\lambda_i(K,\Z^2)$. We assume that
$\lambda_1\lambda_2\vol(K)>9/2$, and let $H$ be a line passing through the
centroid $\vnull$ of $K$, such that $\vnull$ is the midpoint of the
corresponding chord $K\cap H$. Then, by a result of Ehrhart \cite{Ehr} we
know that one of the sets $C^+=(K\cap H^+)\cup\bigl(-(K\cap H^+)\bigr)$ or
$C^-=(K\cap H^-)\cup\bigl(-(K\cap H^-)\bigr)$ is convex, and without loss
of generality we assume that $C^+$ is convex. By the $o$-symmetry of $C^+$
we have, in particular, $\lambda_i\leq\lambda_i(C^+,\Z^2)$, $i=1,2$. Now,
by Theorem~\ref{thm:gruenbaum} and our assumption we get
\[
\vol(K\cap H^+)\geq\frac{4}{9}\vol(K)>
\frac{4}{9}\frac{9}{2\lambda_1\lambda_2}=\frac{2}{\lambda_1\lambda_2},
\]
and therefore,
\[
\vol(C^+)=2\vol(K\cap H^+)
>\frac{4}{\lambda_1\lambda_2}\geq\frac{4}{\lambda_1(C^+,\Z^2)\lambda_2(C^+,\Z^2)},
\]
contradicting Minkowski's inequality \eqref{eq:minkowski}.
\end{proof}

\begin{proposition}\label{prop:ehr_simpl}
Let $S\in\Kgn$ be a simplex and let $\Lambda$ be a lattice of rank $n$. Then
\[
\frac{\vol(S)}{\det\Lambda}\leq\frac{(n+1)^n}{n!}\prod_{i=1}^n\frac1{\lambda_i(S,\Lambda)},
\]
and for $\Lambda=\Z^n$, equality holds for the simplex $T_n$.
\end{proposition}

In ~\cite{Ehr3}, Ehrhart used a nice symmetrization that transforms the
simplex into a parallelepiped. Let $S=\conv\{\vv_0,\vv_1,\ldots,\vv_n\}$
be a simplex with centroid at the origin, that is,
$\vv_0=-\sum_{i=1}^n\vv_i$. We consider the $(n\!-\!1)$-dimensional
subspace $H\subset\R^n$ which is parallel to the facet
$\conv\{\vv_1,\ldots,\vv_n\}$ of $S$. Then $S\cap
H=\conv\{\vw_1,\ldots,\vw_n\}$, where
\[
\vw_i=\vv_0+\frac{n}{n+1}(\vv_i-\vv_0)=\frac1{n+1}\vv_0+\frac{n}{n+1}\vv_i,
\]
$1\leq i\leq n$. Now, we define the parallelepiped
\[
P_H(S)=\conv\left\{\vv_0+\sum_{i=1}^n\e_i(\vw_i-\vv_0):(\e_1,\dots,\e_n)\in\{0,1\}^n\right\}.
\]
The vertex of $P_H(S)$ opposite to $\vv_0$ is
$\vv_0+\sum_{i=1}^n(\vw_i-\vv_0)=-(n-1)\vv_0$. Next, we translate $P_H(S)$
by its center, $\bigl(\vv_0-(n-1)\vv_0\bigr)/2=-(n-2)/2\vv_0$, and we
define the ``symmetral'' of $S$ by $\Pi_H(S)=P_H(S)+(n-2)/2\vv_0$.

\begin{lemma}\label{lem:incl}
Let $S=\conv\{\vv_0,\vv_1,\ldots,\vv_n\}$ have its centroid at the origin
and let $H$ be the $(n\!-\!1)$-dimensional subspace parallel to
$\conv\{\vv_1,\ldots,\vv_n\}$. Let $H^-$ be the halfspace containing the
vertex $\vv_0$. Then
\[
\Pi_H(S)\cap H^-\subseteq\frac{n}2S.
\]
\end{lemma}

\begin{proof}
Since the claim is invariant under affine transformations, we consider the
simplex $T_n$, which can be expressed as
\begin{equation*}
T_n=\left\{(x_1,\dots,x_n)\in\R^n:x_i\geq-1,1\leq i\leq
n,\,\sum_{i=1}^nx_i\leq 1\right\}.
\end{equation*}
Here $H=\bigl\{\vx\in\R^n:\ip{\bone}{\vx}=0\bigr\}$ and $T_n\cap
H=-\bone+\conv\{n\ve_1,\ldots,n\ve_n\}$, and thus the symmetral is
$\Pi_H(T_n)=(n/2)[-1,1]^n$. Since
$H^-=\bigl\{\vx\in\R^n:\ip{\bone}{\vx}\leq 0\bigr\}$, we get from the
facet description of $T_n$ that indeed it holds $\Pi_H(T_n)\cap
H^-\subseteq(n/2)T_n$.
\end{proof}

\begin{proof}[Proof of Proposition \ref{prop:ehr_simpl}]
Again, it suffices to consider the standard lattice $\Lambda=\Z^n$. For
$i\in\{1,\ldots,n\}$, we write
$\lambda_i=\lambda_i\bigl(\Pi_H(S),\Z^n\bigr)$, and let $\vz_i\in\Z^n$ be
such that $\vz_i\in\lambda_i\Pi_H(S)$. Since $\Pi_H(S)$ is $o$-symmetric,
we can assume (after a suitable reflection of $\vz_i$) that
$\vz_i\in\lambda_i\,\bigl(\Pi_H(S)\cap H^-\bigr)$ -- we follow the notation
in Lemma~\ref{lem:incl}. Then Lemma~\ref{lem:incl} implies that
$\vz_i\in\lambda_i(n/2)S$ and hence $\lambda_i(S,\Z^n)\leq(n/2)\lambda_i$,
$1\leq i\leq n$.

Now, we assume that
$\lambda_1(S,\Z^n)\cdot\ldots\cdot\lambda_n(S,\Z^n)\vol(S)>(n+1)^n/n!$. By
definition of $\Pi_H(S)$, we have $\vol\bigl(\Pi_H(S)\bigr)=n!\vol(S\cap
H^-)$ and thus using Theorem~\ref{thm:gruenbaum} we get
\begin{align*}
\vol\bigl(\Pi_H(S)\bigr) & =n!\vol(S\cap H^-)\geq\frac{n!\,n^n}{(n+1)^n}\vol(S)\\
    &>\frac{n^n}{\prod_{i=1}^n\lambda_i(S,\Z^n)}
  \geq\frac{2^n}{\prod_{i=1}^n\lambda_i}.
\end{align*}
It contradicts Minkowski's 2nd theorem, \eqref{eq:minkowski}, and proves
the claim.
\end{proof}

\section{Bounds for the surface area}\label{s:F(K)}

In general, we cannot expect to find upper bounds on $\sur(K)$, or
on the quotient $\sur(K)/\vol(K)$,  in
terms of $\lambda_i(K,\Z^n)^{-1}$ as the
following example shows.

\begin{example}
For $\ell\in\N$, we consider the cross-polytope
\[
K_\ell=\conv\bigl\{\pm\ve_1,\ldots,\pm\ve_{n-1},\pm(\ell\ve_1+\ve_n)\bigr\}.
\]
Then $\lambda_i(K_\ell,\Z^n)=1$, for $1\leq i\leq n$ and all $\ell\in\N$,
but both $\sur(K_\ell)\to\infty$ and $\sur(K_\ell)/\vol(K_\ell)\to\infty$
as $\ell\to\infty$. \label{exam:no_ub_sur}
\end{example}

The proof of the lower bound, i.e., Theorem~\ref{thm:surface}, is based on
the following lemma which might be of independent interest.

\begin{lemma}
Let $Z\in\Z^{n\times n}$, $\det Z\ne 0$, and let $\valpha\in \R^n$ be with
$\Vert\valpha\Vert=1$.
For $\veps\in\R^n$, we write $\valpha_{\veps}=(\e_1\alpha_1,\ldots,\e_n\alpha_n)$.
Then
\begin{equation*}
     \sum_{_{\veps\in\{(\pm1,\ldots,\pm1)^\intercal\}}
     }\Vert Z \valpha_{\veps} \Vert \geq 2^n,
\end{equation*}
and, for $\valpha>\vnull$ equality holds if and only if -- up to column
permutations and signs -- $Z$ is the identity matrix.
\label{lem:matrix}
\end{lemma}

\begin{proof}
After a suitable permutation of the columns $\vz_1,\dots,\vz_n$ of $Z$ we
may assume $\alpha_1\geq\alpha_2\geq\cdots\geq\alpha_n$. Let
$\vg_1,\dots,\vg_n$ be the Gram-Schmidt orthogonal basis associated to
$\vz_1,\dots,\vz_n$, i.e.,
\begin{equation*}
  \vg_i=\vz_i|\lin\{\vz_1,\dots,\vz_{i-1}\}^\perp,\quad 1\leq i\leq n.
\end{equation*}
So $\vg_i$ is the orthogonal projection of $\vz_i$ onto the orthogonal
complement of the $(i\!-\!1)$-dimensional space generated by
$\vz_1,\dots,\vz_{i-1}$. We observe that $\vg_1=\vz_1$. First we claim
that
\begin{equation}
\label{eq:firstclaim}
 \sum_{_{\veps\in\{(\pm1,\ldots,\pm1)^\intercal\}}
     }\Vert Z \valpha_{\veps} \Vert \geq 2^n
     \Vert\alpha_1\vg_1+\alpha_2\vg_2+\cdots +\alpha_n\vg_n\Vert.
\end{equation}
Taking the symmetry of the vectors $\valpha_{\veps}$ into account we have to
show that
\begin{equation*}
\label{eq:secondclaim}
 \sum_{_{\veps\in\{(1,\pm1,\ldots,\pm1)^\intercal\}}
     }\Vert Z \valpha_{\veps} \Vert \geq 2^{n-1}
     \Vert\alpha_1\vg_1+\alpha_2\vg_2+\cdots +\alpha_n\vg_n\Vert.
\end{equation*}
Let $\tilde{\vg}_i=\vz_i|\lin\{\vz_1\}^\perp$, for $2\leq i\leq n$; in
particular, $\tilde{\vg}_2=\vg_2$.  For each
$\veps\in\bigl\{(1,\pm1,\ldots,\pm1)^\intercal\bigr\}$, let
\begin{equation*}
\veps'=(1,-\e_2,\ldots,-\e_n)^\intercal,
\end{equation*}
i.e., the last $n-1$ coordinates change their signs, and let
\[
\vh=\e_2\alpha_2\vz_2+\cdots+\e_n\alpha_n\vz_n.
\]
In view of the properties of Steiner-symmetrization (see e.g.~\cite[Prop.~9.1]{Gr}),
we see that the perimeter of the triangle
$\conv\left\{\pm\alpha_1\vz_1,\vh|\lin\{\vz_1\}^\perp\right\}$ is less
than or equal to the perimeter of $\conv\{\pm\alpha_1\vz_1,\vh\}$, i.e.,
\begin{equation*}
\begin{split}
\Vert Z\valpha_{\veps}\Vert + \Vert Z\valpha_{\veps'}\Vert &
    =\Vert\alpha_1\vz_1 + \vh \Vert + \Vert  \alpha_1\vz_1 - \vh \Vert\\
 & \geq \left\Vert \alpha_1\vz_1 + \vh|\lin\{\vz_1\}^\perp \right\Vert +
    \left\Vert \alpha_1\vz_1 - \vh|\lin\{\vz_1\}^\perp \right\Vert\\
 & =2 \left\Vert\alpha_1\vz_1 + \vh|\lin\{\vz_1\}^\perp \right\Vert\\
 & = 2 \left\Vert  \alpha_1\vg_1+\e_2\alpha_2\vg_2+\e_3\alpha_3\tilde\vg_3+\cdots+
    \e_n\alpha_n\tilde\vg_n\right\Vert.
\end{split}
\end{equation*}
Hence we know that
\begin{align}
\sum_{_{\veps\in\{(1,\pm1,\ldots,\pm1)^\intercal\}}}
 & \Vert Z \valpha_{\veps} \Vert \nonumber\\
 & \hspace*{-5mm}\geq 2\sum_{_{\veps\in\{(1,1,\pm1,\ldots,\pm1)^\intercal\}}}\Vert\alpha_1\vg_1+
  \alpha_2\vg_2+\e_3\alpha_3\tilde\vg_3+\cdots+\e_n\alpha_n\tilde\vg_n \Vert.\label{eq:firstround}
\end{align}
Now we do the same with respect to $\vg_2=\vz_2|\lin\{\vz_1\}^\perp$,
i.e., we consider the points
$\widehat{\vg}_i=\tilde\vg_i|\lin\{\vg_2\}^\perp$, $1\leq i\leq n$, $i\ne
2$. By the orthogonality of $\vg_1,\vg_2$ we have
$\widehat{\vg}_1=\vg_1=\vz_1$, and by the definition of Gram-Schmidt
orthogonal basis we also have $\widehat{\vg}_3=\vg_3$. Arguing as before
but with respect to the triangle
$\conv\{\pm\alpha_2\vg_2,\alpha_1\vg_1+\vh\}$, with
$\vh=\e_3\alpha_3\tilde\vg_3+\cdots+\e_n\alpha_n\tilde\vg_n$, we get
\begin{equation*}
\begin{split}
\Vert\alpha_2\vg_2+\alpha_1\vg_1+\vh\Vert &  + \bigl\Vert\alpha_2\vg_2-(\alpha_1\vg_1+\vh)\bigr\Vert\\
  & \geq 2\left\Vert \alpha_2\vg_2 +(\alpha_1\vg_1+\vh)|\lin\{\vg_2\}^\perp\right\Vert\\
  &  = 2 \left\Vert \alpha_2\vg_2 +  \alpha_1\vg_1 +\e_3\alpha_3\widehat{\vg}_3+\cdots+
    \e_n\alpha_n\widehat{\vg}_n\right\Vert \\
  & = 2 \left\Vert \alpha_1\vg_1 +  \alpha_2\vg_2+\e_3\alpha_3\vg_3+\e_4\alpha_4\widehat{\vg}_4+\cdots+
  \e_n\alpha_n\widehat{\vg}_n\right\Vert,
\end{split}
\end{equation*}
and in view of  the orthogonality of  $\vg_1$  to $\vg_2,\vh$ we conclude
\begin{equation*}
\begin{split}
\Vert\alpha_1\vg_1+\alpha_2\vg_2+\vh \Vert & + \Vert\alpha_1\vg_1 + \alpha_2\vg_2-\vh\Vert \\
 & = \Vert\alpha_2\vg_2 +  \alpha_1\vg_1+\vh\Vert+\bigl\Vert\alpha_2\vg_2 - (\alpha_1\vg_1+\vh)\bigr\Vert \\
 & \geq 2 \left\Vert \alpha_1\vg_1 +  \alpha_2\vg_2+\e_3\alpha_3\vg_3+\e_4\alpha_4\widehat{\vg}_4+\cdots+
  \e_n\alpha_n\widehat{\vg}_n\right\Vert.
\end{split}
\end{equation*}
Hence, together with \eqref{eq:firstround}, we get
\begin{equation*}
\begin{split}
\sum_{_{\veps\in\{(1,\pm1,\ldots,\pm1)^\intercal\}}} & \Vert Z \valpha_{\veps} \Vert \\
 & \hspace*{-1.2cm}\geq 4\sum_{_{\veps\in\{(1,1,1,\pm1,\ldots,\pm1)^\intercal\}}}
    \Vert  \alpha_1\vg_1+\alpha_2\vg_2+\alpha_3\vg_3+\e_4\alpha_4\widehat{\vg}_4+\cdots+
   \e_n\alpha_n\widehat{\vg}_n \Vert.
\end{split}
\end{equation*}
Repeating this procedure we get \eqref{eq:firstclaim} and so it suffices
to show that
\begin{equation*}
   \sum_{i=1}^n\alpha_i^2\Vert\vg_i\Vert^2\geq 1.
\end{equation*}
By the definition of Gram-Schmidt orthogonal basis we have
\begin{equation*}
  Z= (\vg_1,\dots,\vg_n)\,T,
\end{equation*}
where $T$ is an upper triangular matrix whose diagonal elements are
all equal to $1$. Hence for $1\leq i\leq n$ we have
\begin{equation*}
  \det (Z_i^\intercal Z_i) = \Vert\vg_1\Vert^2\cdot\ldots\cdot \Vert\vg_i\Vert^2,
\end{equation*}
where $Z_i$ is the $(n\times i)$-submatrix of $Z$ consisting of the first
$i$ columns; in particular, we have $Z_n=Z$. Let $m_i= \det (Z_i^\intercal
Z_i)$, $1\leq i\leq n$. Then $m_i\in\N$, $m_i\geq 1$, and since
$\|\valpha\|=1$, we may write, using the weighted arithmetic-geometric
mean inequality, that
\begin{equation*}
\begin{split}
\sum_{i=1}^n\alpha_i^2\Vert\vg_i\Vert^2  &
  =\alpha_1^2\,m_1+\alpha_2^2\,\frac{m_2}{m_1}+\alpha_3^2\,\frac{m_3}{m_2}+\cdots + \alpha_n^2\,\frac{m_n}{m_{n-1}}\\
 & \geq m_1^{\alpha_1^2}\left(\frac{m_2}{m_1}\right)^{\alpha_2^2}\cdot\ldots\cdot
  \left(\frac{m_n}{m_{n-1}}\right)^{\alpha_n^2} \\
 & =m_1^{\alpha_1^2-\alpha_2^2}\,m_2^{\alpha_2^2-\alpha_3^2}\cdot\ldots\cdot m_{n-1}^{\alpha_{n-1}^2-\alpha_n^2}\,m_n^{\alpha_n^2}.
\end{split}
\end{equation*}
By assumption we have $\alpha_i\geq \alpha_{i+1}$ and since $m_i\geq
1$ we are done.

If equality holds then we have $m_i=1$, $1\leq i\leq n$, and so
$\Vert\vg_i\Vert = 1$, $1\leq i\leq n$. By the equality discussions of the
Steiner-symmetrization we also know that the vectors $\vz_i$ have to be
pairwise orthogonal and thus $\vg_i=\vz_i$, $1\leq i\leq n$.
\end{proof}

\begin{remark}
We observe that Lemma~\ref{lem:matrix} does not restrict to integer
matrices. It holds for any matrix~$V$ with $\det(V_i^\intercal V_i)\geq
1$, $1\leq i\leq n$, where $V_i$ is the $(n\times i)$-submatrix of $V$
consisting of the first $i$ columns. Equality is attained if only if $V$
is an orthogonal matrix. \label{rem:generalmatrix}
\end{remark}

The connection between the matrix problem
and the surface area  is based on the next calculation.
\begin{fact}\label{obs:facetvolume}
Let $P=\bigl\{\vx\in\R^n:\ip{\va_j}{\vx}\leq b_j, 1\leq j\leq m\bigr\}$ be
a non-redun\-dant representation of a polytope with $\Vert \va_j\Vert =1$,
and let $\phi_j$ be the $(n-1)$-dimen\-sional volume of the facet with
normal vector $\va_j$. Then, for $B\in\GL(n,\R)$, the polytope $BP$ has
outer normal vectors $B^{-\intercal}\va_j$, $1\leq j\leq m$, and the
$(n-1)$-dimensional volume of the facet with normal $B^{-\intercal}\va_j$
is given by
\begin{equation*}
    |\det B|\left\Vert B^{-\intercal}\va_j\right\Vert \phi_j.
\end{equation*}
\end{fact}

Now we are ready to prove the main theorem of this section.

\begin{proof}[Proof of Theorem \ref{thm:surface}]
Let $\vz_i\in\lambda_i\,K\cap\Z^n$, $1\leq i\leq n$, be $n$ linearly independent
lattice points. Let $Z$ be the matrix with columns $\vz_1,\dots,\vz_n$.
Then
\begin{equation*}
Z\,\diag\left(\lambda_1^{-1},\dots,\lambda_n^{-1}\right)C_n^\star
\subseteq K.
\end{equation*}
The volume of each facet of the cross-polytope
$\diag\left(\lambda_1^{-1},\dots,\lambda_n^{-1}\right)C_n^\star$ is
\begin{equation}
\frac{1}{(n-1)!}\sqrt{\sigma_{n-1}(\lambda_1^{-2},\dots,\lambda_n^{-2})},\label{eq:vol_fac_cp}
\end{equation}
and writing
\begin{equation*}
  \valpha=\frac{1}{\sqrt{\lambda_1^2+\cdots+ \lambda_n^2}}\,(\lambda_1,\dots,\lambda_n)^\intercal,
\end{equation*}
the $2^n$ outer unit normal vectors of the above cross-polytope are given
by~$\valpha_{\veps}$, for $\veps\in\{(\pm1,\ldots,\pm1)^\intercal\}$. Hence, in
view of Fact~\ref{obs:facetvolume}, we get
\begin{align}
\sur(K) & \geq \sur\left(Z\,\diag(\lambda_1^{-1},\dots,\lambda_n^{-1})C_n^\star\right)\nonumber\\
  & = \left(\sum_{\veps\in\{(\pm1,\ldots,\pm1)^\intercal\}}
    \!|\det Z|\left\Vert Z^{-\intercal}\valpha_{\veps}\right\Vert\right) \frac{1}{(n-1)!}
    \sqrt{\sigma_{n-1}\!\left(\lambda_1^{-2},\dots,\lambda_2^{-2}\right)}.\label{eqn:surf_proof}
\end{align}
Since $|\det Z| Z^{-\intercal}$ is an integral matrix, the statement of
the theorem follows from Lemma \ref{lem:matrix}, as well as the
characterization of the equality case.
\end{proof}

We notice that Theorem~\ref{thm:surface} together with the upper bound
in~\eqref{eq:minkowski} yields
\[
\frac{\sur(K)}{\vol(K)}\geq\frac{1}{(n-1)!}\sqrt{\lambda_1^2+\dots+\lambda_n^2}.
\]
For $n=2$ this bound improves the one obtained in
\cite{henk1990inequ}, namely, that the ratio $\sur(K)/\vol(K)>\lambda_n$;
moreover, it is tight. But when $n\geq 3$ the above bound is
worse. We conjecture the right bound of this type to be
\begin{equation*}\label{eq:conj_F/vol}
\frac{\sur(K)}{\vol(K)}\geq\sqrt{\lambda_1^2+\dots+\lambda_n^2}.
\end{equation*}
In fact, it would be (asymptotically) sharp, as the example contained in
\cite{henk1990inequ} shows: For the parallelotope
$P_{\mu}=\bigl\{(x_1,\dots,x_n)\in\R^n:|x_i|\leq\mu/2, 1\leq i\leq n-1,\,
|x_n|\leq 1/2\bigr\}$ it is easy to check that
$\lambda_i(P_{\mu},\Z^n)=2/\mu$, $1\leq i\leq n-1$,
$\lambda_n(P_{\mu},\Z^n)=2$, and $\vol(P_{\mu})=\mu^{n-1}$,
$\sur(P_{\mu})=2\mu^{n-2}(n-1+\mu)$. Hence,
\[
\lim_{\mu\to\infty}\frac{\sur(P_{\mu})}{\vol(P_{\mu})\sqrt{\sum_{i=1}^n\lambda_i^2}}=1.
\]
We conclude this section by discussing a generalization of
Theorem~\ref{thm:surface} to arbitrary lattices. For it, we need the
concept of minimal determinants of sublattices: For a lattice $\Lambda$ of
rank $n$, and for $1\leq i\leq n$, we define
\[
 \Dee_i(\Lambda)=\min\left\{\det\Lambda_i:\Lambda_i\textrm{ an }i\textrm{-dimensional sublattice of }\Lambda\right\},
\]
and we write $\Dmin(\Lambda)=\min\left\{\Dee_i(\Lambda)^{1/i}:1\leq i\leq
n\right\}$. Moreover,
\[
 \Lambda^\star=\bigl\{\vx\in\R^n:\ip{\vx}{\vy}\in\Z\textrm{ for all }\vy\in\Lambda\bigr\}
\]
denotes the polar lattice of $\Lambda$. For more information on minimal
determinants we refer to~\cite{schnell} and the references therein.

\begin{theorem}\label{cor:surface_arblat}
Let $K\in\Kon$, $\Lambda$ be a lattice of rank $n$, and let
$\lambda_i=\lambda_i(K,\Lambda)$, $1\leq i\leq n$. Then
\begin{equation*}
   \frac{\sur(K)}{\Dmin(\Lambda^\star)\det\Lambda}\geq \frac{2^{n}}{(n-1)!}
   \sqrt{\sigma_{n-1}\left(\lambda_1^{-2},\dots,\lambda_n^{-2}\right)}.
\end{equation*}
In particular, there exists an absolute constant $c>0$ such that
\begin{equation*}
   \frac{\sur(K)}{\Dee_{n-1}(\Lambda)}\geq \frac{c}{\sqrt{n}}\frac{2^{n}}{(n-1)!}
   \sqrt{\sigma_{n-1}\left(\lambda_1^{-2},\dots,\lambda_n^{-2}\right)}.
\end{equation*}
\end{theorem}
\begin{proof}
Let $\vz_i\in\lambda_i\,K\cap\Lambda$, $1\leq i\leq n$, be $n$ linearly
independent lattice points. Let $Z$ be the matrix with columns
$\vz_1,\dots,\vz_n$. Then, inequality~\eqref{eqn:surf_proof} in the proof
of Theorem~\ref{thm:surface} holds without modification.

In order to apply Lemma~\ref{lem:matrix} in the general setting, we need
to be a bit more careful. Let $B\in\GL(n,\R)$ be a basis of
$\Lambda$, i.e., $\Lambda=B\Z^n$. Then, there exists an integral matrix
$Y\in\Z^{n\times n}$ such that $Z=BY$. Writing $\bar Z=|\det Z|
Z^{-\intercal}$, we get $\bar Z=|\det B|B^{-\intercal}\,|\det
Y|Y^{-\intercal}$, and  $\bar Y=|\det Y|
Y^{-\intercal}$ is an integral matrix. With $\bar Y=(\bar y_1,\ldots,\bar
y_n)$ and using the notation of the proof of Lemma~\ref{lem:matrix}, we
have
\begin{align*}
\sqrt{m_i(\bar Z)}&=\sqrt{\det(\bar Z_i^\intercal \bar Z_i)}\\
 & =\sqrt{|\det B|^{2i}\det\bigl((B^{-\intercal}\bar y_1,\ldots,B^{-\intercal}\bar y_i)^\intercal
    (B^{-\intercal}\bar y_1,\ldots,B^{-\intercal}\bar y_i)\bigr)}\\
 & =(\det\Lambda)^i\det\bigl(\textrm{lattice spanned by }\{B^{-\intercal}\bar y_1,\ldots,B^{-\intercal}\bar y_i\}\bigr).
\end{align*}
Now, since $B^{-\intercal}$ is a basis of the polar lattice
$\Lambda^\star$, the lattice spanned by $\{B^{-\intercal}\bar
y_1,\ldots,B^{-\intercal}\bar y_i\}$ is an $i$-dimensional sublattice of
$\Lambda^\star$, and so we get
\[
\sqrt{m_i(\bar Z)}\geq(\det\Lambda)^i\Dee_i(\Lambda^\star).
\]
Since $m_i(\bar Z)$ is homogeneous of degree $2i$, we need to
multiply~$\bar Z$ by
\[
\left(\det\Lambda\,\min\left\{\Dee_i(\Lambda^\star)^{1/i}:1\leq i\leq
n\right\}\right)^{-1}=\bigl(\Dmin(\Lambda^\star)\det\Lambda\bigr)^{-1}
\]
in order to get a matrix with  $m_i\geq1$, for $1\leq i\leq n$. Hence, by
Remark \ref{rem:generalmatrix}, we get
from \eqref{eqn:surf_proof} that
\begin{equation}
\sur(K)\geq\frac{2^n\Dmin(\Lambda^\star)\det\Lambda}{(n-1)!}
\sqrt{\sigma_{n-1}\left(\lambda_1^{-2},\dots,\lambda_2^{-2}\right)},\label{eqn:surf_proof_al}
\end{equation}
which proves the first claimed inequality.

For the second estimate, let $\Lambda_i$ be an $i$-dimensional sublattice
of~$\Lambda^\star$ with $\det\Lambda_i=\Dee_i(\Lambda^\star)$.
Applying~\eqref{eq:minkowski} to the ball $B_i$ we obtain
\[
 \lambda_1(B_i,\Lambda^\star)^i\leq\lambda_1(B_i,\Lambda_i)^i\leq\frac{2^i}{\vol_i(B_i)}\Dee_i(\Lambda^\star).
\]
In view of $\vol_i(B_i)=\pi^{i/2}/\Gamma(i/2+1)$ and Stirling's
approximation of the $\Gamma$-function, we see that there exists an
absolute constant $c>0$ such that
\[
\Dee_i(\Lambda^\star)^{1/i}\geq\frac{c}{\sqrt{i}}\,\lambda_1(B_i,\Lambda^\star)
=\frac{c}{\sqrt{i}}\min_{\vz\in\Lambda^\star\setminus\{\vnull\}}\|\vz\|=\frac{c}{\sqrt{i}}\Dee_1(\Lambda^\star).
\]
Therefore,
\[
\Dmin(\Lambda^\star)\det\Lambda\geq\frac{c}{\sqrt{n}}\Dee_1(\Lambda^\star)\det\Lambda
=\frac{c}{\sqrt{n}}\Dee_{n-1}(\Lambda),
\]
where the last equality is a well-known relation between a lattice and its
polar lattice (see e.g.~\cite[Cor.~1.3.5]{Martinet}). Using this inequality
and~\eqref{eqn:surf_proof_al}, we obtain the desired estimate.
\end{proof}

\begin{remark}
The investigation of lower bounds for the surface area of a convex body
$K\in\Kgn$ or $K\in\Kn$, in terms of the successive minima of $\D K$,
leads to the question whether there is an analogous statement to
Remark~\ref{obs:cross} for the surface area. We leave this as an open
problem for subsequent studies.
\end{remark}

\section{Bounds for the lattice surface area}\label{s:latticeF(K)}

First of all, we argue that, in general, the lattice surface area cannot
be bounded from below by the successive minima. For a different set of
examples, that exclude lower bounds on the quotient $\lE_{n-1}(P)/\vol(P)$
in terms of the sum of the $\lambda_i(P,\Z^n)$,
see~\cite[Rem.~3.2]{henkschuerwills2005ehrhart}.

\begin{example}\label{exam:no_lb_lsur}
Let $\ell\in\N$ and consider $P_\ell^n=\diag(\ell,1,\ldots,1)C_n^\star$.
Then, the volume of each facet $F$ of $P_\ell^n$ equals
(see~\eqref{eq:vol_fac_cp})
\[
\frac{1}{(n-1)!}\sqrt{\sigma_{n-1}(\ell^2,1,\ldots,1)}=\frac{1}{(n-1)!}\sqrt{1+(n-1)\ell^2}.
\]
Moreover, we have
\[
\det\bigl(\affo(F)\cap\Z^n\bigr)=\bigl\|(1,\ell,\ldots,\ell)\bigr\|=\sqrt{1+(n-1)\ell^2},
\]
and hence $\lE_{n-1}(P_\ell^n)=2^{n-1}/(n-1)!$. So, $\lE_{n-1}(P_\ell^n)$
does not depend on $\ell$, whereas $\lambda_1(P_\ell^n,\Z^n)=1/\ell$ and
$\lambda_2(P_\ell^n,\Z^n)=\cdots=\lambda_n(P_\ell^n,\Z^n)=1$.
\end{example}

Next, we want to prove Theorem~\ref{thm:HSW_nonsymmetric}. It goes along the same
lines as in~\cite[Thm.~1.2]{henkschuerwills2005ehrhart}, and it is based on a
generalized pyramid formula for the volume of a polytope which has been
recently obtained in~\cite{henklinke2014cone}.

\begin{theorem}[Henk \& Linke,~\cite{henklinke2014cone}]\label{thm:cone-volume}
Let $P\in\Kgn$ be a polytope with facets~$F_j$ corresponding to normal
vectors $\va_j$, $1\leq j\leq m$. Furthermore, let $L_k$ be a $k$-dimensional linear subspace. Then
\[
\vol(P)\geq\frac{n}{k}\sum_{\va_j\in
L_k}\vol\bigl(\conv\{\vnull,F_j\}\bigr).
\]
\end{theorem}

\begin{proof}[Proof of Theorem \ref{thm:HSW_nonsymmetric}]
\romannumeral1): First, we observe that the desired inequality is
invariant under translations by rational vectors. Moreover, the centroid
$\cen(P)$ of the rational polytope~$P$ has only rational entries. Indeed,
for any triangulation of $P$ by rational simplices $S_1,\ldots,S_t$, we
have
\[
\cen(P)=\sum_{j=1}^t\frac{\cen(S_j)}{\vol(S_j)}\quad\textrm{ and }
\quad\cen(S_j)=\frac1{n+1}\sum_{v \text{ vertex of } S_j}v\in\Q^n.
\]
So, with $\vol(S_j)\in\Q$, for all $j=1,\ldots,t$, we get
$\cen(P)\in\Q^n$. Therefore, after a suitable translation of $P$ we can
assume that $\cen(P)=\vnull$.

Let $P=\bigl\{\vx\in\R^n:\ip{\va_j}{\vx}\leq b_j,1\leq j\leq m\bigr\}$
with $b_j\in\Q_{>0}$ and primitive normal vectors $\va_j\in\Z^n$, i.e.,
there is no lattice point on the interior of the line segment
$[\vnull,\va_j]$. Writing $F_j$ for the facet of $P$ corresponding to the
normal vector $\va_j$, we have $\Vert
\va_j\Vert=\det\bigl(\affo(F_j)\cap\Z^n\bigr)$, $1\leq j\leq m$. Hence,
\begin{equation}
\lE_{n-1}(P)=\frac12\sum_{j=1}^m\frac{\vol_{n-1}(F_j)}{\det\bigl(\affo(F_j)\cap\Z^n\bigr)}
=\frac12\sum_{j=1}^m\frac{\vol_{n-1}(F_j)}{\Vert
\va_j\Vert}.\label{eq:prf_HSW_ns}
\end{equation}

Writing $\lambda_i=\lambda_i(\D P,\Z^n)$, $1\leq i\leq n$, let
$\vv_1,\ldots,\vv_n\in\D P$ be linearly independent points such that
$\lambda_i\vv_i=\vz_i\in\Z^n$, for every $i=1,\ldots,n$, and let
$L_k=\lin\{\vv_1,\ldots,\vv_k\}$, $1\leq k\leq n$, with $L_0=\{\vnull\}$.

Since the centroid of $P$ lies at the origin, it holds $\D
P\subseteq(n+1)P$ (cf.~\cite[Sect.~34]{bonnesenfenchel1987theory}).
Therefore
\[
\D P\subseteq
\Bigl\{\vx\in\R^n:\bigl|\ip{\va_j}{\vx}\bigr|\leq(n+1)b_j,1\leq j\leq
m\Bigr\},
\]
and since $\vz_i\in\lambda_i\,\D P$, it implies that
\begin{align}
(n+1)\,b_j&\geq\frac1{\lambda_i}\bigl|\ip{\va_j}{\vz_i}\bigr|\quad\textrm{for
all}\quad i=1,\ldots,n\quad\textrm{and}\quad j=1,\ldots,m.\label{eqn:bj_bound}
\end{align}
For $k=0,\dots,n$, we define $V_k=\left\{j:\va_j\in L_k^\perp\right\}$.
Then, $V_0=\{1,\ldots,m\}$ and $V_k\subseteq V_{k-1}$ for each
$k=1,\ldots,n$. Furthermore, let $q$ be the smallest index such that
$V_q=\emptyset$. Then, the integrality of the $\va_j$'s and $\vz_i$'s
gives
\[
(n+1)\,b_j\geq\frac1{\lambda_k}\bigl|\ip{\va_j}{\vz_k}\bigr|\geq\frac1{\lambda_k}\quad\textrm{for
all}\quad j\in V_{k-1}\setminus V_k\quad\textrm{and}\quad k=1,\ldots,q.
\]
So, writing $F_j^o=\conv\{\vnull,F_j\}$, we use~\eqref{eq:prf_HSW_ns}
and get the estimate
\begin{align*}
\lE_{n-1}(P)&=\frac12\sum_{k=1}^q\sum_{j\in V_{k-1}\setminus V_k}\frac{\vol_{n-1}(F_j)}{\Vert \va_j\Vert}\\
&\leq\frac{n(n+1)}2\sum_{k=1}^q\lambda_k\sum_{j\in V_{k-1}\setminus V_k}\frac{\vol_{n-1}(F_j)b_j}{\Vert \va_j\Vert\,n}\\
&=\frac{n(n+1)}2\sum_{k=1}^q\lambda_k\left(\sum_{j\in V_{k-1}}\vol(F_j^o)-\sum_{j\in V_k}\vol(F_j^o)\right)\\
&=\frac{n(n+1)}2\left(\lambda_1\vol(P)+\sum_{k=1}^{q-1}(\lambda_{k+1}-\lambda_k)\sum_{\va_j\in L_k^\perp}\vol(F_j^o)\right).
\end{align*}
In the last equality, we have used that $\sum_{j\in
V_0}\vol(F_j^0)=\vol(P)$ and $V_q=\emptyset$.

Finally, by Theorem~\ref{thm:cone-volume} and the monotonicity of the
successive minima, we obtain
\begin{align*}
\frac{\lE_{n-1}(P)}{\vol(P)}&\leq\frac{n(n+1)}2\left(\lambda_1+\sum_{k=1}^{q-1}\frac{n-k}{n}(\lambda_{k+1}-\lambda_k)\right)\\
&=\frac{n+1}2\left(\sum_{k=1}^{q-1}\lambda_k+(n-q+1)\lambda_q\right)\leq \frac{n+1}2\sum_{k=1}^n\lambda_k.
\end{align*}

\romannumeral2): We argue as above with only minor adjustments: Since for
a polytope $P\in\Kgn$ it holds $P\subseteq -nP$
(see~\cite[Sect.~34]{bonnesenfenchel1987theory}), we may
replace~\eqref{eqn:bj_bound} by
$nb_j\geq\bigl|\ip{\va_j}{\vz_i}\bigr|/\lambda_i$ and note that,
for
$n\geq2$, this can never be an equality if $\ip{\va_j}{\vz_i}\geq 0$. Since
there always exists a $j\in\{1,\ldots,m\}$, with $\ip{\va_j}{\vz_1}>0$, and
hence $j\in V_0\setminus V_1$, we see that there is at least one strict
inequality in the argument. This implies that
$\lE_{n-1}(P)/\vol(P)<(n/2)\sum_{i=1}^n\lambda_i(P,\Z^n)$, as
desired.
\end{proof}

It turns out that there are almost tight examples for Theorem~\ref{thm:HSW_nonsymmetric}~\romannumeral2).

\begin{proposition}\label{prop:ex_hsw_nonsymm}
There are lattice polytopes showing that the factor $n/2$ in
Theorem~\ref{thm:HSW_nonsymmetric}~\romannumeral2) cannot be replaced by a
constant smaller than $n^2/\bigl(2(n+1)\bigr)$.
\end{proposition}

\begin{proof}
Let $P$ be a reflexive $(n-1)$-polytope, i.e., $P\in\mathcal{K}^{n-1}$ is
a lattice polytope containing the origin in its interior such that its
polar $P^\star=\bigl\{\vx\in\R^{n-1}:\ip{\vx}{\vy}\leq1,\vy\in P\bigr\}$
is a lattice polytope as well. Reflexive polytopes were introduced
in~\cite{batyrev} (see~\cite{beyhenkwills2007notes} for more information
and references). We assume further that the centroid of $P$ lies at the
origin, and for $\ell\in\N$, we consider the pyramid
$P_\ell=\conv\bigl\{(n+1)\,\ell P\times\{-1\},n\ve_n\bigr\}$. By
construction, the centroid of $P_\ell$ is the origin as well, and the
intersection of $P_\ell$ with the hyperplane
$\{(x_1,\dots,x_n)\in\R^n:x_n=0\}$ is the polytope~$n\ell P$. Since $P$
has exactly one interior lattice point, we get
\[
\lambda_1(P_\ell,\Z^n)=\dots=\lambda_{n-1}(P_\ell,\Z^n)=\frac{1}{n\ell}\quad
\text{ and }\quad\lambda_n(P_\ell,\Z^n)=\frac1{n},
\]
and therefore,
\begin{equation}
\sum_{i=1}^n\lambda_i(P_\ell,\Z^n)=\frac{\ell+n-1}{n\ell}.\label{eq:P_ell}
\end{equation}
In order to compute the quotient $\lE_{n-1}(P_\ell)/\vol(P_\ell)$, we
consider the pyramid $Q'=\conv\bigl\{Q\times\{0\},\ve_n\bigr\}$ over an
$(n\!-\!1)$-dimensional lattice polytope $Q$. As mentioned in the
introduction, the lattice surface area of $Q'$ is the coefficient of order
$n-1$ of the Ehrhart polynomial
$\#(kQ'\cap\Z^n)=\sum_{i=0}^n\lE_i(Q')k^i$, $k\in\N$, of $Q'$. Therein,
the coefficient $\lE_i(Q')$ is a homogeneous functional of degree~$i$, and
$\lE_n(Q')=\vol(Q')$. Since $Q'$ is a pyramid, we obtain (see
e.g.~\cite[Sect.~2.4]{beckrobins2007computing})
\begin{align*}
\#(kQ'\cap\Z^n)&=\#(kQ\cap\Z^{n-1})+\sum_{j=0}^{k-1}\#(jQ\cap\Z^{n-1})\\
 & =\sum_{j=0}^{n-1}\lE_j(Q)k^j+\sum_{j=1}^n\left(\sum_{i=j-1}^{n-1}\lE_i(Q)\binom{i+1}{j}\frac{b_{i-j+1}}{i+1}\right)k^j,
\end{align*}
where $b_j$ are the Bernoulli numbers. In particular, since $b_0=1$,
$b_1=-1/2$,
\[
\lE_{n-1}(Q')=\frac1{n-1}\lE_{n-2}(Q)+\frac12\lE_{n-1}(Q)=\frac1{n-1}\lE_{n-2}(Q)+\frac12\vol_{n-1}(Q).
\]
Since the polytope $P$ is reflexive, we have
$\lE_{n-2}(P)=\bigl((n-1)/2\bigr)\vol_{n-1}(P)$ (see
\cite[Lem.~3.1]{beyhenkwills2007notes}) and thus using the homogeneity of
the Ehrhart coefficients, we get for $Q=\ell P$,
\[
\lE_{n-1}(Q')=\frac{\ell^{n-2}}{n-1}\lE_{n-2}(P)+\frac{\ell^{n-1}}2\vol_{n-1}(P)=\frac{\ell+1}{2}\ell^{n-2}\vol_{n-1}(P).
\]
Moreover, by $P_\ell=(n+1)Q'-\ve_n$, we obtain
\begin{equation*}
\frac{\lE_{n-1}(P_\ell)}{\vol(P_\ell)}=\frac1{n+1}\frac{\lE_{n-1}(Q')}{\vol(Q')}
=\frac{n}{n+1}\frac{\frac{\ell+1}2\ell^{n-2}\vol_{n-1}(P)}{\vol_{n-1}(\ell
P)}=\frac{n}{2(n+1)}\frac{\ell+1}{\ell}.
\end{equation*}
Hence, an inequality of the type $\lE_{n-1}(P_\ell)/\vol(P_\ell)\leq
c_{n,\ell}\sum_{i=1}^n\lambda_i(P_\ell,\Z^n)$ implies, by~\eqref{eq:P_ell}, that
\[
c_{n,l}\geq\frac{n^2}{2(n+1)}\frac{\ell+1}{\ell+n-1},
\]
which goes to $n^2/\bigl(2(n+1)\bigr)$ when $\ell\to\infty$.
\end{proof}

As mentioned in the introduction, Theorem~\ref{thm:HSW_nonsymmetric}
provides relations between the successive minima of a lattice polytope and
the roots of its Ehrhart polynomial (for more information
see~\cite{beckrobins2007computing,henkschuerwills2005ehrhart, OhsugiShibata} and the
references inside). To this end, for a lattice polytope $P\in\Kn$, we
regard the right hand side of~\eqref{eq:ehrhartpoly} as a polynomial in a
complex variable $s\in\C$, and write
\[
\sum_{i=0}^n\lE_i(P)s^i=\prod_{i=1}^n\left(1+\frac{s}{\gamma_i(P)}\right).
\]
Hence, $-\gamma_1(P),\ldots,-\gamma_n(P)$ are the roots of the Ehrhart
polynomial of~$P$ and we have
$\lE_i(P)=\sigma_i\bigl(1/\gamma_1(P),\ldots,1/\gamma_n(P)\bigr)$, $1\leq
i\leq n$. This implies that
$\lE_{n-1}(P)/\vol(P)=\sum_{i=1}^n\gamma_i(P)$, and therefore
Theorem~\ref{thm:HSW_nonsymmetric} can be reformulated as follows (for
convenience we include the result for $o$-symmetric lattice polytopes that
already appeared in~\cite{henkschuerwills2005ehrhart}).

\begin{corollary}\label{cor:roots_appl}
Let $P\in\Kn$ be a lattice polytope.
\begin{enumerate}[i)]
 \item Then
\[
\sum_{i=1}^n\gamma_i(P)\leq\frac{n+1}{2}\sum_{i=1}^n\lambda_i(\D
P,\Z^n),
\]
and the standard simplex $S_n$ shows that the inequality is best possible.
 \item If $P\in\Kgn$ and $n\geq2$, then
\[
\sum_{i=1}^n\gamma_i(P)<\frac{n}{2}\sum_{i=1}^n\lambda_i(P,\Z^n),
\]
and Proposition~\ref{prop:ex_hsw_nonsymm} shows that the factor $n/2$ is of the right order.
 \item If $P\in\Kon$, then
 \[
\sum_{i=1}^n\gamma_i(P)\leq\frac{1}{2}\sum_{i=1}^n\lambda_i(P,\Z^n),
\]
and equality is attained, for example, by $C_n$ and $C_n^\star$.
\end{enumerate}
\end{corollary}


\begin{thebibliography}{99}

\bibitem{batyrev} V. V. Batyrev, Dual polyhedra and mirror symmetry for Calabi-Yau hypersurfaces in toric varieties,
{\it J. Algebraic Geom.} {\bf 3} (3) (1994), 493--535.

\bibitem{beckrobins2007computing} M. Beck, S. Robins, {\it Computing the continuous discretely}.
Undergraduate Texts in Mathematics, Springer, New York, 2007.

\bibitem{BermanBerndtsson} R. J. Berman, B. Berndtsson, The volume of
K\"ahler-Einstein Fano varieties and convex bodies, {\it
arXiv:1204.1308v1}.

\bibitem{betkehenk1993approx} U. Betke, M. Henk, Approximating the volume of convex bodies,
{\it Discrete Comp. Geom.} {\bf 10} (1) (1993), 15--21.

\bibitem{Betke:1993jn} U. Betke, M. Henk, J. M. Wills,
Successive-minima-type inequalities, {\it Discrete Comp. Geom.} {\bf 9}
(2), (1993), 165--175.

\bibitem{beyhenkwills2007notes} C. Bey, M. Henk, J. M. Wills, Notes on the
roots of Ehrhart polynomials, {\it Discrete Comp. Geom.} {\bf 38} (1)
(2001), 81--98.

\bibitem{bonnesenfenchel1987theory} T. Bonnesen, W. Fenchel, {\it Theorie
der konvexen K\"orper}. Springer, Berlin, 1934, 1974. English translation:
{\it Theory of convex bodies}. Edited by L. Boron, C. Christenson and B.
Smith. BCS Associates, Moscow, ID, 1987.

\bibitem{Ehr} E. Ehrhart, Une g\'en\'eralisation du th\'eor\`{e}me de
Minkowski, {\it C. R. Acad. Sci. Paris} {\bf 240} (1955), 483--485.

\bibitem{Ehr2} E. Ehrhart, Une g\'en\'eralisation probable du th\'eor\`{e}me fondamental de Minkowski,
{\it C. R. Acad. Sci. Paris} {\bf 258} (1964), 4885--4887.

\bibitem{Ehr3} E. Ehrhart, Volume r\'eticulaire critique d'un simplexe,
{\it J. Reine Angew. Math.} {\bf 305} (1979), 218--220.

\bibitem{groemer1966zus} H. Groemer, Zusammenh\"angende Lagerungen konvexer K\"orper,
{\it Math. Z.} {\bf 94} (1966), 66--78.

\bibitem{Gr} P. M. Gruber, {\it Convex and Discrete Geometry}.
Springer, Berlin Heidelberg, 2007.

\bibitem{GL} P. M. Gruber, C. G. Lekkerkerker, {\it Geometry of
Numbers}. North Holland, Amsterdam, 1987.

\bibitem{gruenbaum1960parts} B. Gr{\"u}nbaum, Partitions of
mass-distributions and of convex bodies by hyperplanes, {\it Pacific J.
Math.} {\bf 10} (1960), 1257--1261.

\bibitem{henk1990inequ} M. Henk, Inequalities between successive minima and
intrinsic volumes of a convex body, {\it Monatsh. Math.} {\bf 110} (1990),
279--282.

\bibitem{henkcifre2008intrinsic} M. Henk, M. A. Hern\'andez Cifre,
Intrinsic volumes and successive radii, {\it J. Math. Anal. Appl.} {\bf
343} (2) (2008), 733--742.

\bibitem{henklinke2014cone} M. Henk, E. Linke, Cone-volume measures of
polytopes, {\it Adv. Math.} {\bf 253} (2014), 50--62.

\bibitem{henkschuerwills2005ehrhart} M. Henk, A. Sch{\"u}rmann, J. M.
Wills, Ehrhart polynomials and successive minima, {\it Mathematika} {\bf
52} (1-2) (2006), 1--16.

\bibitem{Malikiosis:2010tb} R. D. Malikiosis, An optimization problem
  related to Minkowski's successive minima, {\it Discrete Comp. Geom.}
  {\bf 43} (4) (2010), 784--797.

\bibitem{Malikiosis:2010vi} R. D. Malikiosis, A discrete analogue for
Minkowski's second theorem on successive minima, {\it Adv. Geom.} {\bf 12}
(2) (2012), 365--380.

\bibitem{Martinet} J. Martinet, {\it Perfect lattices in Euclidean spaces},
Springer-Verlag, Berlin, 2003.

\bibitem{Milman:2000th} V. D. Milman, A. Pajor,  Entropy and asymptotic
geometry of non-symmetric convex bodies, {\it Adv. Math.} {\bf 152} (2)
(2000), 314--335.

\bibitem{minkowski1896geometrie} H. Minkowski, {\it Geometrie der Zahlen}.
Teubner, Leipzig-Berlin, 1896, reprinted by Johnson Reprint Corp., New York, 1968.

\bibitem{NillPaffenholz} B. Nill, A. Paffenholz, On the equality case in
Ehrhart's volume conjecture, {\it To appear in Adv. Geom.}

\bibitem{OhsugiShibata} H. Ohsugi, K. Shibata, Smooth Fano polytopes whose
Ehrhart polynomial has a root with large real part, {\it Discrete Comp.
Geom.} {\bf 47} (3) (2012), 624--628.

\bibitem{Schneider:2014td} R. Schneider, {\it Convex Bodies: The
Brun-Minkowski Theory}, Cambridge University Press, 2nd expanded edition,
2014.

\bibitem{schnell} U. Schnell, Minimal determinants and lattice
inequalities, {\it Bull. London Math. Soc.} {\bf 24} (1992), 606--612.

\bibitem{Schnell1993} U. Schnell, J. M. Wills, On successive minima and
intrinsic volumes, {\it Mathematika}, {\bf 40} (1) (1993), 144--147.

\bibitem{Ziegler:1995gl} G. M. Ziegler, {\it Lectures on polytopes},
Springer-Verlag, 1995, revised sixth printing 2006.

\end{thebibliography}
\end{document}